\documentclass{imsart}

\RequirePackage[OT1]{fontenc}
\RequirePackage{amsthm,amsmath,amsfonts}
\RequirePackage[numbers]{natbib}
\RequirePackage[colorlinks,citecolor=blue,urlcolor=blue]{hyperref}
\RequirePackage{hypernat}
\usepackage[dvips]{graphics}

\newtheorem{theorem}{Theorem}[section]
\newtheorem{lemma}{Lemma}[section]
\newtheorem{corollary}{Corollary}[section]

\newtheorem{proposition}{Proposition}[section]

\newtheorem{remark}{Remark}
\newtheorem{assumption}{Assumption}
\newtheorem*{lemmaa}{Lemma A}
\newtheorem*{lemmab}{Lemma B}
\newtheorem*{lemmac}{Lemma C}

\newcommand{\mI}{{\mathcal I}}
\newcommand{\mJ}{{\mathcal J}}
\newcommand{\mK}{{\mathcal K}}

\newcommand{\mQ}{{\mathcal Q}}

\newcommand{\mY}{{\mathcal Y}}
\newcommand{\bN}{{\mathbb N}}
\newcommand{\bR}{{\mathbb R}}
\newcommand{\bZ}{{\mathbb Z}}
\newcommand{\bP}{{\mathbb P}}
\newcommand{\bE}{{\mathbb E}}

\newcommand{\mC}{{\mathcal C}}

\begin{document}
\begin{frontmatter}
\title{Flow level convergence and insensitivity for multi-class queueing networks.}
\runtitle{Convergence and insensitivity for queueing networks}
\author{\fnms{Neil Stuart} \snm{Walton}\ead[label=e1]{n.s.walton@statslab.cam.ac.uk}}
\affiliation{University of Cambridge}
\address{Centre for Mathematical Sciences\\
Wilberforce Road\\
Cambridge\\
CB3 0AJ, U.K.\\
\printead{e1}}
\runauthor{N.S. Walton}
\begin{abstract}
We consider a multi-class queueing network as a model of packet transfer in a communication network. We define a second stochastic model as a model document transfer in a communication network where the documents transferred have a general distribution. 
We prove the weak convergence of the multi-class queueing process to the document transfer process. 
Our convergence result allows the comparison of general document size distributions, and consequently, we prove general insensitivity results for the limit queueing process.
\end{abstract}
\end{frontmatter}
\section{Introduction}
We present a result that formally demonstrates the separation of timescales between a communication model, where discrete packets are transferred, and a second model, where documents are transferred elastically. Such convergence results are distinct from the fluid and diffusion limit results which are typically applied to queueing processes. The result applies to the transfer of documents with generally distributed sizes and quasi-reversible queues. This extends the convergence proof which was previously applied to the simpler case of exponential file sizes and processor sharing queues \cite{Wa09}. By generalising this result, we can formally prove insensitivity results about the limit queueing system. 

Our prelimit model is a quasi-reversible multi-class queueing network as considered by Baskett et al. \cite{BCMP75}, Kelly \cite{Ke79}. We endow this model with a specific routing structure. Documents for transfer on different routes of the network, arrive as a Poisson process. A document consists of a number of packets which are transferred one-by-one across their route.

Our limit model is a bandwidth sharing model. These stochastic processes model the elastic transfer of documents in a communication network. Bandwidth sharing models were first introduced by Roberts and Massouli\'e \cite{MaRo98}. In this paper, we are particularly interested in bandwidth sharing models associated with multi-class queueing networks. Models of this form were first considered by Bonald and Proutiere \cite{BoPr04}. Bonald and Proutiere demonstrated that these models were insensitive to documents with a phase-type distributions. Our convergence result will allow us to extend this result to documents of any non-atomic distribution.


The formal convergence proof uses a coupling argument. The proof demonstrates weak convergence in the Skorohod topology of the number of documents in transfer of the multi-class queueing network to that of a bandwidth sharing model. The prelimit models considered here are well understood product form queueing networks. Even so such explicit product form results are not required, the arguments used to prove this separation of timescales result are general and could be applied in the analysis of a diverse range of queueing models.

\subsection{An informal description of the results}
We consider a multi-class queueing network. The queueing network processes documents along different routes. Arriving documents are divided into packets which are sent across the network one-by-one. Once all the packets in a document are sent the document departs. 

We could describe the state of this queueing network in several ways: we could consider the explicit behaviour of the network, $Q$, by storing residual document sizes and the location of packets on their routes; or, we could consider the states of packets only $\tilde{Q}$, in doing so, we would ignore information about the residual sizes of documents; or, finally, we could consider the residual document sizes only $Y$, and thus ignore precise information about the positions of packets on their routes. These descriptions form an explicit description, a packet-level description and a flow-level description of our network. 

We are interested in the interactions of a queueing network at these levels. In particular, the rate at which packets are transferred in a modern communication network is often an order of magnitude larger than the time it takes to transfer a document. Thus, given the number of documents in transfer, we should be able to abstract away the packet level behaviour of the network. More formally, conditional on the number of documents in transfer, the quick transition of packets within the network should imply that the distribution of packets with in the network converges quickly to its stationary distribution, and thus the processing of documents is best described by the stationary behaviour of the packet-level queueing network.

We mathematically demonstrate this by taking a sequence of multi-class queueing networks $Q^{(c)}$. Along this sequence, we increase the rate that queues process packets by a factor $c$ and accordingly increase the sizes of documents by a factor $c$. In this regime, packet transitions occur on a time scale of order $O(1/c)$, whilst document transfers remain at a timescale of $O(1)$. Thus in this limit, in between document arrival-departure events, the packet-level state $\tilde{Q}$ will converge to its stationary distribution and document transfers will receive a linear rate of transfer given by the stationary throughput of the packet-level network.

In this paper, we show that in this limit. If we let the initial state and document sizes converge, then the times at which documents arrive and depart the network will converge. The formal result, Theorem \ref{spinning converge} proves that the flow-level state of queueing networks converges in the Skorohod topology to a flow-level model of the network.

With this result we can consequently prove a general insensitivity result for these flow-level networks. In this context, insensitivity means that the stationary distribution of number of documents in transfer only depends on the mean size of documents. We can demonstrate this property for discrete document sizes in the prelimit networks. The Skorohod convergence implies that the stationary distribution for the prelimit network converges to the stationary distribution for the limit network. Consequently in Corollary \ref{Ch2: insensitive coroll}, we demonstrate that the limit queueing network is insensitive amougst all non-atomic document size distributions.

\subsection{Organisation}
In Section \ref{net struct}, we give basic notation used throughout the paper. In Section \ref{multi-class networks}, we review results on some well understood product form queueing networks. In Section \ref{Bandwidth networks}, we define the bandwidth sharing networks which will be the limit of our mutli-class queueing networks. We, also, define what it means for these networks to be insensitive. In Section \ref{sec:spinning conv}, we prove the main convergence result Theorem \ref{spinning converge}. Finally, in Section \ref{Ch2: Insensitivity Sec}, we prove the insensitivity of these queueing networks.

\section{Notation and network structure}\label{net struct}
We let the finite set $\mJ$ index the set of \textit{queues} in a network. Let $J=|\mJ|$. A \textit{route} through the network is a non-empty set of queues. Let $\mI\subset 2^\mJ$ be the set of routes. Let $I=|\mI|$. For each route $i=\{j^i_1,...,j^i_{k_i}\}\in\mI$, we associate an order $(j^i_1,...,j^i_{k_i})$. We allow for queues to be repeated in our route order. For $i\in\mI$ and $j\in i$, we let $\zeta_{ji}\in\bN$ be the number of times queue $j$ is included in ordering $(j^i_1,...,j^i_{k_i})$.  Also we define the set of queue-route incidences, $\mK:=\{(j,i):i\in\mI, j\in\mJ, j\in i\}$ and let $K=|\mK|$. We will view a multi-class queueing network model as transferring a number of documents across the different routes of the network.  The vector $n=(n_i: i\in\mI)\in\bZ_+^I$ will denote the number of documents in transfer across the routes of the network. We also let the vector $m=(m_{ji}:(j,i)\in\mK)\in\bZ_+^K$ refer to the number of packets in transfer across each route at each queue. That is $m_{ji}$ is the number of packets on route $i$ at queue $j$. We also define the number of packets in transfer at a queue to be
\begin{equation*}
m_j:=\sum_{i: j\in i} m_{ji}, \qquad j\in\mJ.
\end{equation*}
For each $n\in\bZ_+^I$, we define $S(n)=\{ m\in \bZ_+^K: \sum_{j:j\in i}m_{ji}=n_i\;\forall\, i\in\mI\}$, that is the set of queue sizes with $n$ documents in transfer on each route.

\section{Multi-class queueing networks}\label{multi-class networks}
In this section, we present some well understood queueing networks that will be studied subsequently. In order to model the transfer of documents across a packet switching network, we define a special case of these queueing networks where packets have a specific routing structure. We then the define closed queueing networks as described in \cite[Section 3.4]{Ke79}. 

\subsection{Multi-class queue}\label{Multi-class queue}
First, we define what we will call a \textit{multi-class queue}. We consider a single queue $j$ from a set of queues $\mJ$. We call the customers of this queue packets. The queue will receive packet arrivals from different classes. The set of classes will consist of a set of packet route choices $\mC$.\footnote{Although for now we choose $\mC$ to be arbitrary, we will later consider $\mC=\{i\in\mI: j\in i\}$, the set of routes using queue $j$.} Packets occupy different positions within a queue. Given there are $m_j\in\bZ_+$ packets at queue $j$ packets may occupy positions $1,2,...,m_j$. Packets of each route at the queue require an independent exponentially distributed service requirement with mean $1$. Given there are $m_j\in\bZ_+$ packets at queue $j$, the total service devoted to packets is given by $\phi_j(m_j)$. We assume $\phi_j(m_j)>0$ if $m_j>0$. This service is then divided amongst packets within the queue.  Given there are $m_j\in\bZ_+$ packets at queue $j$, a proportion $\gamma_j(l,m_j)$ of service is devoted to the packet in position $l\in\{1,...,m_j\}$ of queue $j$. Since $\gamma_j(\cdot,m_j)$ represents a proportion,
\begin{equation*}
\sum_{l=1}^{m_j}\gamma_j(l,m_j)=1,\qquad  m_j\in\bN.
\end{equation*}
Upon completing its service the packet at position $l$ will leave the queue and the packets at positions $l+1,...,m_j$ will move to positions $l,...,m_j-1$, respectively. 

We assume packets of class  $c\in\mC$ will arrive at the queue from  independent Poisson processes of rate $\rho_{jc}$. Given there are $m_j$ packets at the queue an arriving packet will move to position $l\in\{1,...,m_j+1\}$ with probability $\delta_j(l,m_j+1)$. Once again as $\delta_j(\cdot,m_j+1)$ represents a proportion
\begin{equation*}
\sum_{l=1}^{m_j+1}\delta_j(l,m_j+1)=1,\qquad  m_j\in\bZ_+.
\end{equation*}
When a packet arrives at position $l$ the packets in positions $l,...,m_j$ will move to positions $l+1,...,m_j+1$, respectively.

Let $q^j=(c_1^j,...,c_{m_j}^j)\in\mI^{m_j}$, for $m_j>0$, give the state of queue $j$. Let function, $T^c_{\cdot, (j,l)}$ denote the arrival of a class $c$ packet to position $l$ in queue $j$ and let function $T^c_{(j,l),\cdot}$ denote the departure of a class $c$ packet in position $l$. Thus the state of this queue forms a continuous-time Markov chain with transition rates given by,
\begin{equation*}
 r(q^j,q'^j)=
\begin{cases}
 \rho_c \delta_j(l,m_j+1) &\text{for } q'^j=T^c_{\cdot, (j,l)} q^j,\; l=1,...,m_j+1,\; c\in\mC\\
\phi_j(m_j)\gamma_j(l,m_j) & \text{for } q'^j=T^c_{(j,l),\cdot} q^j,\; c^j_l=c,\; l=1,...,m_j,\\
0 & \textit{otherwise}.
\end{cases}
\end{equation*}

The queue itself will not discriminate between different packet's classes and thus the stationary distribution of the queue size will be oblivious to different packets' route type. Ignoring packet classes, when stationary $M_j$ the Markov chain recording the total number of packets at the queue is reversible.
Given $m_j$, routes of the packets in positions $1,2,...,m_j$ are independent.
The probability a packet in a given position is from route $i$ is $\frac{\zeta_{ji}\rho_i}{\sum_{r\ni j} \zeta_{jr}\rho_r}$. 
Thus letting Markov chain $Q_j$ record the position and routes of packets at queue $j$ and letting  $\mQ_j=\cup_{m_j=1}^{\infty} \mI^{m_j}$ gives all possible states of the queue, we can calculate the stationary distribution of the queue.

\begin{proposition}[BCMP \cite{BCMP75}, Kelly \cite{Ke75}]\label{single server queue ed}
A stationary multi-class queue is quasi-reversible and the stationary distribution of $Q_j$ must be
\begin{equation}\label{explicit queue ed}
 \bP_j(Q_j=(c(1),...,c(m_j)))=\frac{1}{B_j} \prod_{l=1}^{m_j}\frac{\rho_{c(l)}}{\phi_j(l)},\qquad  (c(1),...,c(m_j))\in\mQ_j.
\end{equation}
Moreover, for $j\in\mJ$, the process $(M_{jc}:\: i\in\mC)$ giving the number of packets of each route type at queue $j$, has stationary distribution
 \begin{equation}\label{Ch1:queue ed}
 \bP(M_{jc}=m_{jc}, \forall\, c\in\mC )=\frac{1}{B_j}\left(\begin{array}{cc} m_j\\m_{jc}\;:i\in\mC \end{array}\right)\frac{\prod_{c\in\mC}({\rho_{c}})^{m_{jc}}}{\prod_{l=1}^{m_j} \phi_j(l)}
\end{equation}
$\forall\, (m_{jc}:\; c\in \mC)$, where we define 
\begin{equation}\label{combi term}
\left(\begin{array}{cc} m_j\\m_{jc}\;:c\in \mC \end{array}\right)=\frac{m_j!}{\prod_{c\in\mC} (m_{jc}!)}.
\end{equation}
\end{proposition}

The combinatorial term in (\ref{combi term}) is required as the probability distribution (\ref{Ch1:queue ed}) ignores the order of packets within the queue.

\subsection{Multi-class queueing networks}\label{multi spinning}
A \textit{multi-class queueing network with spinning} (MQNwS) is a multi-class network of quasi-reversible queues with the following class routing structure. The class of a packet is of the form $c=(i,k,y)$ where $i\in\mI$ records the route the packet is on, $k\in\bN$ records the stage of the packet on route $i$ and $y\in\bN$ records the packet's residual document size, that is the remaining number of times the packet must traverse its route.
Routing through classes occurs in the following way. If route $i$ has associated route order $(j^i_1,...,j^i_{k_i})$ then as a Poisson process of rate $\nu_i\bP(X_i=x)$ class $(i,1,x)$ packets arrive at queue $j^i_1$. Here $X_i$ is a random variable with values in $\bN$ and with mean $\mu^{-1}_i<\infty$. 
For $k=1,...,k_i-1$, a class $(i,{k},y)$ packet on departing queue $j^i_{k}$ will join queue $j^i_{k+1}$ and become a class $(i,{k+1},y)$ packet.
For a packet that has completed service at the final queue on route $i\in\mI$ and has not been fully processed through the network, that is a packet of class $(i,{k_i},y)$ with $y>1$, the packet will join queue $j^i_1$ as a class $(i,1,y-1)$ packet. For a route $i\in\mI$ packet that has completed its service at the final queue $k_i$ and has been fully processed through the network, that is of class $(i,k_i,1)$, the packet will depart the network. In addition, we let the constant $\zeta_{ji}\in\bZ_+$ give the number of times a packet visits queue $j$ each time it traverses route $i$. Finally, we define traffic intensities $\rho_i=\frac{\nu_i}{\mu_i}$ for each $i\in\mI$.

We can interpret this routing structure in two ways. First, we could consider each packet on route $i$ to arrive as a Poisson process and to repeat its route a number of times that is independent and with distribution equal to $X_i$. This interpretation leads us to think of a packet as spinning around its route a random number of times. Second, we could consider the network to be transferring documents. Documents which require to be transferred across route $i$ arrive as a Poisson process as of rate $\nu_i$. Each document consists of a number of packets, that is independent and with distribution equal to $X_i$. These packets are then sent across the network one by one until the document is transferred. 



For a Markov process description of a MQNwS we record its \textit{explicit state}:  we let $q=(q_j: j\in\mJ)$, where $q_j=(c_j(1),...,c_j(m_j))$ gives the class of each customer in each occupied position in queue $j$. Here the class $c_j(l)=(i_j(l),k_j(l),y_j(l))$ records the route, stage and residual document size associated with the $l$-th packet in queue $j$.  We let $\mQ$ define the set of all possible states for this explicit description of our queueing network. 

Recalling that $m_{ji}$ is the number of route $i$ packets in transfer at queue $j$ and that $n_i$ is the number of route $i$ documents in transfer. As each document has one packet in transfer in the network at any point in time
\begin{equation*}
 n_i=\sum_{j\in\mJ} m_{ji}, \qquad i\in\mI.
\end{equation*}

We define two further descriptions of the state of a MQNwS: the packet level state and the flow level state.

We define the \textit{packet level state} of a multi-class queueing network with spinning to be, $\tilde{q}=(\tilde{q}_j:\; j\in\mJ)$, where $\tilde{q}_j=(\tilde{c}_j(1),...,\tilde{c}_j(m_j))$ and where $\tilde{c}_j(l)=(i_j(l),k_j(l))$ records the route and stage associated with the $l$-th packet in queue $j$.  We let $\tilde{\mQ}$ define the set of all possible packet level states for this description of our queueing network. The packet level state of a MQNwS is concerned with the position and route of packets but not of the state of document transfer. Similarly the flow level state is interested in the state of document transfer and not in the specific position of packets.

We define the \textit{flow level state} of a MQNwS to be, given by vector
\begin{equation*}
 y=( y_{ik} : k=1,...,n_i, i\in\mI )
\end{equation*}
Here, we order elements so that $y_{ik}\leq y_{ik+1}$ for all $k=1,...,n_i-1$. Note, we record no information about each packet's position on its route.
As described above $n_i$ refers to the number of route $i$ documents in transfer on route $i$ and $k$ indexes each specific packet in transfer on route $i$.
  The number $y_{ik}$ is the residual document size of the $k$-th document in route $i$, that is the number of packets yet to be transferred from the document.
We let $\mY$ be the set of flow level states achievable by a MQNwS.

The processes associated with the packet level or flow level state of a multi-class queueing network with spinning need not be Markov. However, these state descriptions will be useful for proving weak convergence results. For this purpose we define on $\mY$ the norm
\begin{equation}\label{Ch2: Y norm}
 ||y-y'||=
\begin{cases}
 \max_{i\in\mI} \max_{k=1,...,n_i} |y_{ik}-y'_{ik} | & \text{if}\quad n_i=n'_i,\quad \forall i\in\mI,\\
\infty &\text{otherwise.}
\end{cases}
\end{equation}

\subsubsection{Stationary behaviour}
We now calculate certain quantities associated with the stationary distribution of a MQNwS. As a direct consequence of known reversibility results \cite[Theorem 3.1]{Ke79}, we can calculate the stationary distribution of a MQNwS.

\begin{theorem}\label{multi-spin ed}
The explicit state of an ergodic multi-class queueing network with spinning has stationary distribution,
\begin{equation}\label{explicit ed}
 \bP(Q=q)=\prod_{j\in\mJ} \frac{1}{B_j} \prod_{l=1}^{m_j}\left(\frac{\nu_{i_j(l)}\bP(X_{i_{j}(l)}\geq y_{j}(l))}{\phi_j(l)}\right),\qquad  q\in\mQ, 
\end{equation}
provided
\begin{equation} \label{B stability cond}
B_j=\sum_{m_j=1}^{\infty} \left(\prod_{l=1}^{m_j}\frac{\sum_{i: j\in i} \zeta_{ji}\rho_i}{\phi_j(l)}\right)<\infty, \qquad \forall j\in\mJ.
\end{equation}
\end{theorem}
\begin{proof} 
A multi-class queueing network with spinning is a network of quasi-reversible queues with a deterministic routing structure. It is known, \cite[Theorem 3.1]{Ke79}, that a network of quasi-reversible queues has a stationary distribution 
\begin{equation*}
\bP(Q=q)= \prod_{j\in\mJ} \bP(Q_j=q_j)
\end{equation*} 
where,
\begin{equation*}
\bP(Q_j=q_j)=\frac{1}{B_j} \prod_{l=1}^{m_j} \frac{\beta_{jc_j(l)}}{\phi_j(l)},\qquad j\in\mJ,
\end{equation*}
and where $\beta_{jc}$ solves the traffic equations
\begin{equation*}
\beta_{jc}=\nu_{jc}+ \sum_{l,d} \beta_{l,d} p_{ld,jc},\qquad j\in\mJ, c\in\mC.
\end{equation*}
Here $\nu_{jc}$ is the arrival rate of class $c$ customers at queue $j$ and $p_{jc,ld}$ gives the packet routing probabilities, which in our case are, for $c=(i,k,y)$,
\begin{align*}
p_{jc,ld}=
\begin{cases}
1 & \text{if } k < k_i,\; d=(i,k+1,y)\text{ and } l= j_{k+1}^i,\\
1 & \text{if } k=k_i,\; d=(i,1,y-1),\; y>0\text{ and }l=j^i_1,\\
1 & \text{if } y=1,\; k=k_i,\; (l,d)=\cdot,\\
0 & \text{otherwise}.
\end{cases}
\end{align*}
In this way, packets are transferred between queues, the next packet is injected at the ingress and document departures occur. 

So, all that is needed is to verify that $\tilde{\beta}_{j,c}$ solves the traffic equations along our deterministic path. Observe that, for $k>1$, $\tilde{\beta}_{j,(i,k,y)}=\tilde{\beta}_{j,(i,k-1,y)}$ and, for $k=1$
\begin{align*}
\tilde{\beta}_{j,(i,1,y)}=\nu_i\bP(X_i\geq y_i) &= \nu_i \bP(X_i=y) + \nu_i\bP(X_i\geq y-1)\\
&= \nu_{i,(i,1,y)}+ \tilde{\beta}_{j,(i,k_i,y+1)}.
\end{align*}
 This verifies the traffic equations are satisfied and hence gives the result. 

\end{proof}
The condition (\ref{B stability cond}) is the necessary and sufficient for a multi-class queueing network with spinning to be ergodic and thus is equivalent to the assumptions ergodicity in subsequent results. We encapsulate this in the following assumption:

\begin{assumption} \label{MQNwS assump}
Unless stated otherwise we assume a multi-class queueing network with spinning satisfies the following necessary and sufficient condition for ergodicity:
\begin{equation} \label{B stability cond 1}
B_j=\sum_{m_j=1}^{\infty} \left(\prod_{l=1}^{m_j}\frac{\sum_{i: j\in i} \zeta_{ji}\rho_i}{\phi_j(l)}\right)<\infty, \qquad \forall j\in\mJ.
\end{equation}
\end{assumption}

The following three corollaries are a consequence of Theorem \ref{multi-spin ed}. Each of these results require summing over an appropriate set of states. For example, from Theorem \ref{multi-spin ed}, we can calculate the stationary distribution of the number of packets in transfer along each route at each queue.
\begin{corollary}\label{CH1: M ed thrm}
Given the stability condition, Assumption \ref{MQNwS assump}, $M=(M_{ji}:(j,i)\in \mK)$, the number of packets in transfer across each route at each queue, has stationary distribution,
\begin{equation}\label{Ch2:ed}
\bP(M=m)=\prod_{j\in\mJ}\left(\frac{1}{B_j}\left(\begin{array}{cc} m_j\\m_{ji}\;:i\ni j \end{array}\right)\frac{\prod_{i: j\in i}(\zeta_{ji}\rho_i)^{m_{ji}}}{\prod_{l=1}^{m_j}\phi_j(l)}\right)
\end{equation}
\end{corollary}
\begin{proof}
We know from Theorem \ref{multi-spin ed} that our queueing network has a stationary distribution equal to that of a simpler queueing network. In this simpler queueing network, each queue $j$ behaves independently in isolation and where class $c=(i,k,y)$ packets, with $j^i_k=j$, arrive at queue j as a Poisson process of rate $\nu_i\bP(X_i\geq y)$. Let us work with this simpler but equivalent queueing model. In this model, route $i$ packets arrive into queue $j$ as a Poisson process of rate
\begin{equation*}
 \sum_{k: j_k^i=j}\sum_{y=1}^{\infty} \nu_i \bP(X_i\geq y) = \zeta_{ji}\rho_i.
\end{equation*}
So queue $j$ will have an independent stationary distribution exactly of the form of (\ref{explicit queue ed}). So, as in Section \ref{Multi-class queue}, by ignoring packet positions we can gain distribution (\ref{Ch1:queue ed}) for each queue. Thus by independence equation (\ref{Ch2:ed}) holds for the network.
\end{proof}
\begin{remark}
Observe that distribution (\ref{Ch2:ed}) only depends on the distribution of $X_i$ through its mean $\frac{1}{\mu_i}$. Thus the stationary distribution of a MQNwS only depends on the distribution of document sizes through their mean size. This suggests a form of insensitivity holds. This point is noted by Massouli\'e and in the thesis of Proutiere \cite{Pr03}. Similar observations are made earlier in \cite{Ke79} for individual queues. We will discuss this observation in more detail in Chapter 3.
\end{remark}

We can express the stationary distribution of the number of documents in transfer on each route, $N=(N_i:i\in\mI)$. We define $\mQ(n)$ and $\tilde{\mQ}(n)$ be the set of explicit states and packet level states a MQNwS that occur with positive probability given there are $n\in\bZ_+^I$ documents in transfer on each route. We also let $S(n)=\{m\in\bZ_+^K: \sum_{j:j\in i} m_{ji}=n_i,\;\; \forall\, i\in\mI \}$ be the set of route-queue states achievable given there are $n\in\bZ_+^I$ documents in transfer.
\begin{corollary}\label{CH1: N ed thrm}
Given the stability condition Assumption \ref{MQNwS assump} is satisfied. For a MQNwS, $N=(N_i:i\in\mI)$ the number of documents in transfer on each route has stationary distribution
\begin{equation}\label{ed N}
\bP(N=n)=\frac{B_n}{B} \prod_{i\in\mI} \rho_i^{n_i},\qquad n\in\bZ_+^I,
\end{equation}
where we define
\begin{align}
&\qquad \qquad B:=\prod_{j\in\mJ} B_j,\qquad & \label{B}\\
B_n&:=\sum_{m\in S(n)}\prod_{j\in\mJ}\left(\left(\begin{array}{cc} m_j\\m_{ji}\;:i\ni j
\end{array}\right)\left(\frac{\prod_{i: j \in i}\zeta^{m_{ji}}_{ji} }{\prod_{l=1}^{m_j}\phi_j(l) }\right)\right),\qquad &\;n\in\bZ_+^I.\label{bn}
\end{align}
\end{corollary}

Finally, we give the stationary distribution for the packet level state of a MQNwS, $\tilde{Q}=(\tilde{Q}_j:j\in\mJ)$. And, we also give the stationary distribution the flow level state of a MQNwS, $Y=(Y_{ik}:\; k=1,...,N_i,\; i\in\mI)$.

\begin{corollary}\label{Ch2: packet flow corol}
Given the stability condition Assumption \ref{MQNwS assump} is satisfied, the stationary packet level state of a MQNwS, $\tilde{Q}=(\tilde{Q}_j:j\in\mJ)$, has distribution
\begin{equation*}
 \bP(\tilde{Q}=\tilde{q})=\prod_{j\in\mJ} \frac{1}{B_j} \prod_{l=1}^{m_j}\frac{\rho_{i_j(l)}}{\phi_j(l)},\qquad  \tilde{q}\in\tilde{\mQ}. 
\end{equation*}
The stationary flow level state of a MQNwS, $Y=(Y_{ik}:\; k=1,...,N_i,\; i\in\mI)$, has distribution
\begin{equation}\label{Ch2:ed Y}
\bP(Y=y)=\frac{B_n}{B}\prod_{i\in\mI}\left( \begin{array}{cc} n_i \\ n_{iy}\: : \: y\in\bN \end{array}\right) \prod_{y\in\bN} \big(\nu_i \bP(X_i\geq y_{ik})\big)^{n_{iy}},\qquad  y\in\mY.
\end{equation}
Here $n_{iy}$ is the number of route $i$ packets with residual file size $y$. Also we define,
\begin{equation*}
 \left( \begin{array}{cc} n_i \\ n_{iy}\: : \: y\in\bN \end{array}\right):= \frac{n_i!}{ \prod_{y\in\bN} (n_{iy}!)}.
\end{equation*}
\end{corollary}

The above two corollaries simply involve summing distribution \eqref{explicit ed} over the specified set of states. We omit the explicit calculation in their proof.

\subsection{Closed multi-class queueing network}\label{closed queue}
A \textit{closed multi-class queueing network} behaves as an MQNwS except that document arrivals and departures are forbidden, see \cite[Section 3.4]{Ke79}. In effect the network behaves as if there are a fixed number of infinitely large documents in transfer. We now more formally define a closed multi-class queueing network.

Given there are $n\in\bZ^I_+$ packets on each route, a closed multi-class queueing network is a packet level Markov process on the states $\tilde{Q}(n)$ 
We now define the class and routing structure of this queueing network. The class of a packet is of the form $\tilde{c}=(i,k)$ where $i\in\mI$ records the route a packet is on and $k\in\bN$ records the stage of the packet on its route $i$. 
Routing through classes occurs in the following way. For $k=1,...,k_i-1$, a class $(i,{k})$ packet on departing queue $j^i_{k}$ will join queue $j^i_{k+1}$ and become a class $(i,{k+1})$ packet. 
A class $(i,{k_i})$ packet that has completed service will join queue $j^i_1$ as a class $(i,1)$ packet. 

This description is sufficient to give a Markov chain description of a closed multi-class queueing network, but is not sufficient for this Markov chain to be irreducible. For example, a network consisting of a single last-come-first-served queue would reducible. For this reason, we require the following assumption to hold throughout this paper.

\begin{assumption}\label{assump}
We assume for all closed queueing networks in this thesis that the set of states $\tilde{\mQ}(n)$ is irreducible.
\end{assumption}

\noindent It is worth noting that if Assumption \ref{assump} is broken then there need not be a unique stationary distribution or a unique stationary throughput for the closed queueing network. Note due to the finite state space of these Markov chains we do not require any stability condition to hold.

As is proven in Section 3.4 of \cite{Ke79}, we now give the stationary distribution for this queueing network.

\begin{corollary}\label{closed corol}
Given Assumption \ref{assump}, for a closed multi-class queueing network with $n\in\bZ_+^I$ documents in transfer across routes, the number of packets in transfer on each route at each queue has stationary distribution
\begin{equation}\label{closed ed}
\bP_n(M=m)=\frac{1}{B_n}\prod_{j\in\mJ}\left(\left(\begin{array}{cc} m_j\\m_{ji}\;:i\ni j
\end{array}\right)\left(\frac{\prod_{i: j \in i}\zeta^{m_{ji}}_{ji} }{\prod_{k=1}^{m_j}\phi_j(k) }\right)\right),
\end{equation}
for each $m\in S(n)$, where $B_n$ is defined by (\ref{bn}).
\end{corollary}

Finally, we can characterise the stationary throughput of packets in a closed multi-class queueing networks.

\begin{corollary}\label{throughput}
Given Assumption \ref{assump}, for a closed multi-class queueing network with $n\in\bZ_+^I$ documents in transfer across routes and with $n_i>0$, the stationary throughput of each route $i$ packet, at stage $k$ and at queue $j=j^i_k$ is
\begin{equation*}
\frac{1}{n_i}\frac{B_{n-e_i}}{B_n},
\end{equation*}
where $B_n$ is defined by (\ref{bn}) and $e_i$ is the $i$-th unit vector in $\bR_+^I$. 
\end{corollary}
\begin{proof}
 The probability the network is in state $m\in\bZ_+^K$ is given by (\ref{closed ed}). Given the network is in state $m$, by Corollary 3.4 of \cite{Ke79} or from stationary distribution (\ref{closed ed}), the probability at queue $j$ the packet position $k'\in\{1,...,m_j\}$ is traversing route $i$ at stage $k$ is $\frac{1}{\zeta_{ji}}\frac{m_{ji}}{m_j}$. The throughput of the packet in position $k'$ of queue $j$ is $\gamma_j(k',m_j)\phi_j(m_j)$. By our irreducibility assumption, all arrangements of the $n_i$ route $i$ packets are equally likely. Thus the probability this packet is any specific route $i$ packet is $\frac{(n_i-1)!}{n_i!}=\frac{1}{n_i}$. Thus, the stationary throughput of this route $i$ packet is
\begin{align*}
&\:\sum_{\substack{m\in S(n):\\ m_{j}>0}} \sum_{k'=1}^{m_j}\gamma_j(k',m_j)\phi_j(m_j)\frac{1}{\zeta_{ji}}\frac{m_{ji}}{m_j}\frac{1}{n_i}\frac{1}{B_n}\prod_{l\in\mJ}\left(\left(\begin{array}{cc} m_l\\m_{lr}\;:r\ni l
\end{array}\right)\left(\frac{\prod_{r: l \in r}\zeta^{m_{lr}}_{lr} }{\prod_{c=1}^{m_l}\phi_l(c) }\right)\right)\\
&=\sum_{\substack{m\in S(n):\\ m_{j}>0}}\phi_j(m_j)\frac{1}{\zeta_{ji}}\frac{m_{ji}}{m_j}\frac{1}{n_i}\frac{1}{B_n}\prod_{l\in\mJ}\left(\left(\begin{array}{cc} m_l\\m_{lr}\;:r\ni l \end{array}\right)\left(\frac{\prod_{r: l \in r}\zeta^{m_{lr}}_{lr} }{\prod_{c=1}^{m_l}\phi_l(c) }\right)\right)\\
&=\sum_{m'\in S(n-e_i)}\frac{1}{n_i}\frac{1}{B_n}\prod_{l\in\mJ}\left(\left(\begin{array}{cc} m'_l\\m'_{lr}\;:r\ni l \end{array}\right)\left(\frac{\prod_{r: l \in r}\zeta^{m'_{lr}}_{lr} }{\prod_{c=1}^{m'_l}\phi_l(c) }\right)\right)=\frac{1}{n_i}\frac{B_{n-e_i}}{B_n}.
\end{align*}
In the first inequality, we used the fact that $\sum_{l=1}^{m_j}\gamma_j(l,m_j)=1,\; \forall\, m_j\in\bN$. In the second equality, we cancelled terms and substituted $m'_{lr}=m_{lr}-1$ if $(l,r)=(j,i)$ and $m'_{lr}=m_{lr}$ otherwise.
\end{proof}
In subsequent chapters, an important quantity will be
\begin{equation*}
\Lambda^{SN}_i(n)=\frac{B_{n-e_i}}{B_n},
\end{equation*}
the stationary rate packets are transferred on route $i$ of a closed multi-class queueing network.

\section{Bandwidth sharing networks}\label{Bandwidth networks}

In this section, we consider a flow level bandwidth sharing model introduced by Massouli\'e and Roberts \cite{MaRo98, MaRo99}. We call these models \textit{stochastic flow level models} (SFLM). SFLMs model the dynamic, elastic transfer rate received by document transfers in a communication network. 
Multi-class queueing networks with spinning (MQNwS) model packet level dynamics SFLMs model document level dynamics.
We think of MQNwSs as a microscopic model of a communication network.
We think of SFLMs as a macroscopic model of a communication network.
We will formally relate MQNwS and SFLMs.

Massouli\'e and Roberts \cite{MaRo99} discuss the separation of time scales between a certain SFLM and MQNwS.  In the next section, we will give a proof that a SFLM is the limit of a sequence of MQNwS, and thus, we formally justify a separation of time scales. We call our limit flow level model a ``spinning network''. The models of this type are considered by Bonald and Proutiere \cite{BoPr04} under the name the ``Store-Forward Network''.

In performing this analysis, we are able to prove general insensitivity results for the spinning network. As cited by Proutiere \cite[Section 3.4]{Pr03} the spinning network was first considered by Massouli\'e because of its insensitivity. Bonald and Proutiere \cite{BoPr04} proved insensitivity for spinning networks with documents with size given by phase type distributions.

In this section, we introduce the stochastic flow level models and we define the spinning network. In the next section, and specifically in Theorem \ref{spinning converge}, we prove the main result of this chapter: the weak convergence of a sequence multi-class queueing networks to its spinning network. In Section \ref{Ch2: Insensitivity Sec} and specifically in Corollary \ref{Ch2: insensitive coroll}, we prove insensitivity results which hold as a consequence of Theorem \ref{spinning converge}.

\subsection{Bandwidth allocations and stochastic flow level models}\label{sec:sflm}\label{gsflm}
A \textit{bandwidth allocation policy} is a map $\Lambda:\bZ_+^I\rightarrow \bR_+^I$. For $n\in\bZ_+^I$, the vector $\Lambda(n)=(\Lambda_i(n):i\in\mI)$ is a \textit{bandwidth allocation}. Here $\Lambda_i(n)$ represents the rate that route $i$ documents are transferred through each route of a communication network, given there are $n=(n_i:i\in\mI)$ documents in transfer on each route. 


The stochastic model we describe represents the randomly varying number of document transfers within a network. The model is studied as a flow level model of Internet congestion control. We first assume that documents have a size that is exponentially distributed. We will then generalise to document sizes that are independent and of a general distribution.

For document sizes that are independent exponentially distributed, a \textit{stochastic flow level model operating under bandwidth allocation policy $\Lambda$} (SFLM)  is a continuous-time Markov chain on $\bZ_+^I$ with rates
\begin{align}\label{sflm}
q(n,n')=
\begin{cases}
\nu_i & \text{if }\; n'=n+e_i,\\
\mu_i\Lambda_i(n) & \text{if }\; n'=n-e_i\text{ and }n_i>0,\\
0 & \text{otherwise},
\end{cases}
\end{align}
for $n,n'\in\bZ_+^I$, where $e_i$ is the $i$-th unit vector in $\bZ_+^I$.

This model can be interpreted as follows: documents wishing to be transferred across route $i$ arrive as a Poisson process of rate $\nu_i$. These documents have a size that is independent and exponentially distributed with mean $\mu_i^{-1}$. If currently the number of documents in transfer across routes is given by vector $n\in\bZ_+^I$ then each document on route $i$ is transferred at rate $\frac{\Lambda_i(n)}{n_i}$. Documents are then processed at this rate until there is a change in the network's state, either by a document transfer being completed and thus leaving the network, or by a document arrival occurring. Thanks to the memoryless property of our process we need not record residual document sizes when an arrival or departure event occurs. 

The key distinction between this model of document transfer and our previous queueing models is that we do not consider packet level dynamics. These dynamics are abstracted away, and instead, we only consider the flow-level descriptions of the network's state.

We can generalise SFLMs to allow the transfer of documents with an independent arbitrarily distributed size. In this case, similar to the flow level state of a multi-class queueing network with spinning, we record the \textit{flow level state} of a generalised SFLM. For each document in transfer, we will record the documents residual size, that is the amount of the document that is still to be processed. Given there are $n=(n_i:i\in\mI)$ documents in transfer, the flow level state of a generalised stochastic flow level model is given by the vector
\begin{equation}\label{flow description}
 y=(y_{ik}: k=1,...,n_i, i\in\mI). 
\end{equation}
Here $y_{ik}\in (0,\infty)$ is the residual document size of the $k$-th document in transfer on route $i$. We order elements so that $y_{ik} \leq y_{ik+1}$ for $k=1,...,n_i-1$. 

The dynamics of this generalised SFLM are morally the same as our previous definition: documents arrive as a Poisson process; documents are transferred at an elastic rate depending on the number of documents in transfer along different routes and documents depart once transferred. 

More explicitly, the dynamics of this model are defined as follows. Documents for transfer on route $i$ arrive as a Poisson process of rate $\nu_i$. An arriving document on route $i$ will then have a residual document size $X'_i$ added to the flow level description \eqref{flow description}. We assume $X'_i$ is an independent positive random variable with finite mean $\mu_i^{-1}$, and we assume $X'_i$ is equal in distribution to some positive random variable $X_i$. In between a document arrival or departure event, the residual document size of a route $i$ document decreases linearly at rate $\frac{\Lambda_i(n)}{n_i}$. A document on route $i$ departs the network at the instant its residual document size equals $0$. At this point, the corresponding document is removed from the network's flow level state description.

Given the network's state $y$, all future events are a function of $y$ and independent random variables, thus the state description describes this process as a Markov process. As described in Section \ref{multi spinning}, we let $\mY$ be the set of flow level states.

\subsection{Insensitive stochastic flow level models}

A stochastic flow level model, as described above, has stationary distribution $\pi_Y$ when
\begin{equation*}
 \pi_Y(A)=\bP(Y(0)\in A)\quad  \text{implies}\quad   \pi_Y(A)=\bP(Y(t)\in A),
\end{equation*}
 $\forall\:A\in {\mathcal B}(\mY)$ and $\forall$ $t\in\bR_+$. Here ${\mathcal B}(\mY)$ is the Borel $\sigma$-field defined on the set of flow level states $\mY$ from norm (\ref{Ch2: Y norm}). 
We say that a random variable $X$ with values in $\bR_+$ is \textit{non-atomic} if $\bP(X=x)=0$ for all $x\in\bR_+$.

 We say that a stochastic flow level model is \textit{insensitive} to non-atomic distributions with stationary distribution $\pi_N=(\pi_N(n):n\in\bZ_+^I)$, if every generalised SFLM with non-atomic document size distributions, mean document sizes $(\frac{1}{\mu_i}: i\in\mI)$ has a stationary distribution $\pi_Y$ satisfying $$\pi_N(n)=\bP_{\pi_Y}(N(0)=n),\qquad \forall\: n\in\bZ_+^I.$$ In other words, the distribution $\pi_N$ only depends on the document size distribution through its mean $(\mu^{-1}_i:i\in\mI)$.

\subsection{Spinning networks}\label{Ch2: Spinning Networks}
Bandwidth allocations represent the stationary rate of document transfer, given the number of documents in transfer on each route. From Corollary \ref{throughput}, we can define a bandwidth allocation that represents the stationary behaviour of a MQNwS. We define a \textit{spinning allocation} to be the stationary throughput of a closed multi-class queueing network. That is for each $\forall\, n\in\bZ_+^I$, we define
\begin{equation}
\Lambda^{SN}_i(n):=\frac{B_{n-e_i}}{B_n},\label{spinning allocation}
\end{equation}
where $e_i$ is the $i$-th unit vector in $\bR_+^I$ and $B_n$ is defined by (\ref{bn}). The stochastic flow level model defined by a spinning allocation policy is called \textit{a spinning network}.
Proutiere \cite{Pr03} notes that this bandwidth allocation is first defined by Laurent Massouli\'e. Insensitivity results on this bandwidth allocation are explored in Bonald and Proutiere\cite{BoPr04}.

\section{Convergence of open queueing networks to spinning networks}\label{sec:spinning conv}

We are now in a position to prove the main results of this paper. The stochastic flow level models of \cite{MaRo98} are intended to represent the flow level dynamics of document transfer in a packet switched network. The aim of this section is to formally justify this interpretation for spinning networks. As a consequence of this analysis, we are able to formally prove insensitivity of spinning networks.

For exponential document sizes and processor sharing queues of fixed capacity, it has been demonstrated that a series of multi-class queueing networks converged weakly to the spinning network in the Skorohod topology \cite{Wa09}. In this section, we generalize theses argument to include general document size distributions and for the general queueing networks discussed in Section \ref{multi spinning}. Although our proof is applied to networks of quasi-reversible queues, the proof applied is phrased so that a more diverse range of queueing processes could be considered. In this sense we generalize Theorem 3.1 \cite{Wa09}, whose proof is specific to the specific queueing and document sizes considered.

\subsection{Limit and prelimit parameters}
For our limit model, we consider the stochastic flow level model for the spinning network. We assume documents have a general positive distribution. As discussed in Section \ref{gsflm}, we assume documents for transfer on route $i\in\mI$ have a distribution given by positive random variable $X^{(\infty)}_i$, with finite mean $\mu_i^{-1}$. We let process $Y^{(\infty)}=(Y^{(\infty)}(t)\in\mQ:t\in\bR_+)$ give the flow level state of this generalised stochastic flow level model and we let $N^{(\infty)}=(N^{(\infty)}(t)\in\bZ_+^I: t\in\bR_+)$ give the number of documents in transfer on each route of the spinning network.

For our prelimit model, we consider a sequence of multi-class queueing networks with spinning indexed by $c\in\bN$. For this sequence, we assume that the parameters for queues $\mJ$, routes $\mI$, route orders $(j_1^i,...,j_{k_i}^i)$ and Poisson arrival rates $\nu=(\nu_i: i\in\mI)$ are all fixed and coincide with the same parameters used to define our spinning network. In our sequence of multi-class queueing networks with spinning, we choose to vary the number of packets in each document and the rate at which packets are transferred through the network. For the $c$-th multi-class queueing network, we let route $i$ document's size have a distribution $X^{(c)}_i$ such that
\begin{equation*}
 \frac{X^{(c)}_i}{c}\Rightarrow X^{(\infty)}_i,\quad \text{as}\quad c\rightarrow\infty,\quad  i\in\mI,
\end{equation*}
and we vary the queueing capacities so that $\phi^{(c)}_j(\cdot)=c\phi_j(\cdot)$, for $j\in\mJ$. 

Our choice of scalings are purposefully chosen so that transitions between queues occur at a time scale of order $O(\frac{1}{c})$ and thus the number of transitions before a document departure is of order $O(1)$. See Figure \ref{Ch2: figure 2} for further explanation.

\begin{figure}\label{Ch2: figure 2}
\scalebox{0.7}[0.7]{\includegraphics{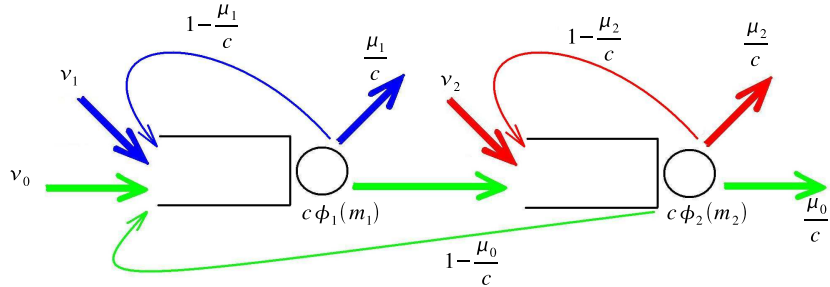}}
\caption{The picture above gives the transition rates for a multi-class queueing network with spinning. Queues process packets at rates given by $c\phi_1(m_1)$ and $c\phi_2(m_2)$. Thus the rate packets are transferred between queues is of order $O(c)$. Documents on routes 0, 1 and 2 arrive as a Poisson processes of rates $\nu_0, \nu_1, \nu_2$. Therefore documents arrive at a rate of order $O(1)$. In this example, documents on routes 0, 1 and 2 have a geometric distribution with parameters $\mu_0/c$, $\mu_1/c$ and $\mu_2/c$, respectively. Now consider the rate documents depart the network. For route 0, for example, the rate documents depart is of the order of $c\phi_2(m_2)\times \mu_0/c=\mu_0\phi_2(m_2)$. Thus document departures occur at a rate of order $O(1)$. This justifies a separation of timescales between packet transfer and document transfer. This separation of timescales will be required to form a limit process from a sequence of multi-class queueing networks with spinning. }
\end{figure}

For $c\in\bN$, we let process $Q^{(c)}=(Q^{(c)}(t)\in\mQ:t\in\bR_+)$ give the explicit queueing description of the $c$-th multi-class queueing network with spinning and we let process $N^{(c)}=(N^{(c)}(t)\in\bZ_+^I: t\in\bR_+)$ give the number of documents in transfer on each route of the $c$-th network. Let $Y^{(c)}=(Y^{(c)}(t): t\in\bR_+)$ and $\tilde{Q}^{(c)}=( \tilde{Q}^{(c)}(t):\: t\in\bR_+)$ be the respective processes corresponding to the flow level state and packet level state of the $c$-th multi-class queueing network with spinning. 

Associated with the multi-class queueing network with spinning $Q^{(1)}$, we will consider $\bar{Q}_n=(\bar{Q}_n(t): t\in\bR_+)$ the closed multi-class queueing network with $n\in\bZ_+^I$ packets on each route. We make the following assumption about each $\bar{Q}_n$

\begin{assumption}\label{assump 2}
We assume Assumption \ref{assump} holds for $\bar{Q}_n$ for all $n\in\bZ_+^I$. That is $\bar{Q}_n$ is an irreducible Markov chain for all $n\in\bZ_+^I$.
\end{assumption}

As noted in Section \ref{closed queue} this assumption excludes reducibility issues which can occur in closed queueing networks where a queue serves a single deterministically chosen packet.

We will also require an assumption on the spinning network $Y^{(\infty)}$.

\begin{assumption}\label{assump 3}
We assume for $Y^{(\infty)}$ that, almost surely, there are no simultaneous document arrival-departure events.
\end{assumption}

This assumption avoids complications associated with the definition of convergence in the Skorohod Topology. Later we will verify that if distribution $X_i^{(\infty)}$ is non-atomic $\forall\, i\in\mI$ then Assumption \ref{assump 3} holds.\footnote{Given the Poisson arrival process of this model Assumption \ref{assump 3} should hold for all SFLMs, provided the initial distribution is chosen so that documents do not arrive or depart at the same time.} 

Our main theorem, Theorem \ref{spinning converge}, considers weak convergence on bounded time intervals. Thus, we do not require assumptions on the networks long run behaviour, such as Assumption \ref{MQNwS assump}, however we will subsequently require some assumptions for results on insensitivity.

\subsection{Theorem and proof}
We now introduce and prove the main result.

\begin{theorem} \label{spinning converge}
For $c\in\bN$, take an multi-class queueing network with spinning $Q^{(c)}$, as described above. Assume Assumptions \ref{assump 2} and \ref{assump 3} hold for each $c\in\bN$. Let $Y^{(\infty)}$ denote the flow level state of the spinning network, as described above. If the initial flow level state converges,
\begin{equation}\label{Ch2: inital Y}
\frac{Y^{(c)}(0)}{c}\Rightarrow Y^{(\infty)}(0)\quad\text{as}\quad c\rightarrow\infty
\end{equation}
then, for each $T>0$, the stochastic processes converge in the Skorohod topology on interval $[0,T]$
\begin{equation*}
\frac{Y^{(c)}}{c}\Rightarrow Y^{(\infty)}\quad\text{as}\quad c\rightarrow\infty.
\end{equation*}
\end{theorem}
\begin{proof}[\textbf{Proof of Theorem \ref{spinning converge}}]
We will prove this result using a coupling argument. In between document arrival-departure events, a MQNwS behaves as a closed queueing network. We couple MQNwS so that in between arrival-departure events this closed queueing network behaviour is determined by a single closed queueing process. By doing this, Skorohod convergence results become a consequence of renewal theory results.

We split the proof into four sections. In the first section, we couple the queueing network's initial states. In the second section, we state an induction hypothesis which we will use to prove weak convergence. In the third section, we form a coupling of our queueing networks. In the fourth section, we prove this coupling satisfies the induction hypothesis. Finally in the fifth section, we prove weak convergence in the Skorohod topology.\\
\textbf{Coupling the initial state:}\\
We start by coupling the initial state of our process. By (\ref{Ch2: inital Y}) and the Skorohod Representation Theorem \cite[Section 6]{Bi99} we may choose a sequence of coupled random variables $\{Y^{(c)}(0)\}_{c\in \bN\cup \{ \infty\}} $ such that, almost surely
\begin{equation} 
 \frac{Y^{(c)}(0)}{c} \xrightarrow[c\rightarrow \infty]{} Y^{(\infty)}(0). \label{CH2: coupled inital Y}
\end{equation}
For $c\in\bN$ and given our coupled sequence $\{Y^{(c)}(0)\}_{c\in \bN\cup \{ \infty\}}$ we know the required distribution of $Q^{(c)}(0)$. We may choose a sequence of functions $f^{(c)}: {\mathcal Y}\times [0,1]\rightarrow \mQ$ such that,
\begin{equation*}
 \bP(f^{(c)}(y,U)=q)=\bP(Q^{(c)}(0)=q|Y^{(c)}(0)=y),\qquad \forall\, q\in\mQ,\; y\in\mY
\end{equation*}
where here $U$ is an independent uniform random variable on $[0,1]$. Thus from a single uniform random variable and the coupled sequence $\{Y^{(c)}(0)\}_{c\in \bN}$, we may define the coupled initial state of each MQNwS by
\begin{equation}
 Q^{(c)}(0)=f^{(c)}(Y^{(c)}(0),U),\qquad c\in\bN. \label{CH2: coupled inital Q}
\end{equation}
\textbf{Induction Hypothesis:}\\
We now inductively construct our coupled process under the following induction hypothesis on $\kappa\in\bZ_+$. For $c\in\bN\cup \{\infty\}$, let $\tau^{\kappa,(c)}$ be the $\kappa$-th document arrival-departure event for the flow level state of our coupled process $Y^{(c)}$, $c\in\bN\cup \{\infty\}$. We assume under this induction hypothesis that we have already defined $Y^{(c)}$ on the interval $[0,\tau^{\kappa,(c)}]$ and that under this coupling
\begin{align}
 \tau^{k,(c)} &\xrightarrow[c\rightarrow\infty]{} \tau^{k,(\infty)}\label{Ch2: tau kappa conv},\qquad k=0,...,\kappa,\\
\frac{Y^{(c)}(\tau^{k,(c)})}{c}&\xrightarrow[c\rightarrow\infty]{} Y^{(\infty)}(\tau^{k,(\infty)}),\qquad k=0,...,\kappa. \label{Ch2: Y kappa conv}
\end{align}
Our induction hypothesis states that there exists a coupling of $Y^{(c)}$ extended to the next arrival-departure event, i.e. on the interval $(\tau^{\kappa,(c)},\tau^{\kappa+1,(c)}\wedge T]$, such that,
\begin{align}
 &\sup_{t\in (\tau^{\kappa,(c)},\tau^{\kappa+1,(c)}\wedge T]} |\lambda^{\kappa,(c)}(t)-t |\xrightarrow[c\rightarrow\infty]{} 0\label{Ch2: time converge} \\
 &\sup_{t\in (\tau^{\kappa,(c)},\tau^{\kappa+1,(c)}\wedge T]} \bigg|\bigg| Y^{(\infty)}(\lambda^{\kappa,(c)}(t))-\frac{Y^{(c)}(t)}{c}\bigg|\bigg| \xrightarrow[c\rightarrow \infty]{} 0, \label{Ch2: space converge}
\end{align}
where $\lambda^{\kappa, (c)}:[\tau^{\kappa,(c)},\tau^{\kappa+1,(c)}\wedge T]\rightarrow [\tau^{\kappa,(\infty)},\tau^{\kappa+1,(\infty)}\wedge T] $ is the function that linearly interpolates between $\lambda^{\kappa, (c)}(\tau^{\kappa,(c)})= \tau^{\kappa,(\infty)}$ and $\lambda^{\kappa, (c)}(\tau^{\kappa+1,(c)}\wedge T)= \tau^{\kappa+1,(\infty)}\wedge T$. Here norm $||\cdot ||$ is defined by (\ref{Ch2: Y norm}). This completes the statement of the induction hypothesis.

Note taking $\tau^{0,(c)}=0$, $\forall\, c\in\bN$ by (\ref{CH2: coupled inital Y}) and (\ref{CH2: coupled inital Q}) our induction hypothesis holds for $\kappa=0$. Also note the convergence statements (\ref{Ch2: time converge}) and (\ref{Ch2: space converge}) are stronger than (\ref{Ch2: tau kappa conv}) and (\ref{Ch2: Y kappa conv}). 
\textbf{Coupling:}\\
Given our induction hypothesis holds until time $\tau^{\kappa,(c)}$, we will define a coupling until the next arrival-departure time $\tau^{\kappa+1,(c)}$. In order to simplify notation, we will use the shorthand $\tilde{q}^{\kappa,(c)}=\tilde{Q}^{(c)}(\tau^{\kappa,(c)})$, $y^{\kappa,(c)}=Y^{(c)}(\tau^{\kappa,(c)})$ and $n^{\kappa,(c)}=N^{(c)}(\tau^{\kappa,(c)})$. These denote the packet level state, flow level state and number of documents in transfer for the $c$-th network at time $\tau^{\kappa,(c)}$.

By assumption (\ref{Ch2: tau kappa conv}), $\exists\, c'$ such that $\forall\, c>c'$
\begin{equation}
n^{\kappa,(c)}=n^{\kappa,(\infty)}. \label{Ch2: n to n}
\end{equation}
Let $\bar{Q}^{\kappa}$ define a closed multi-class queueing network with $n^{\kappa,(\infty)}$ packets across each route and with queue service capacities defined by $(\phi_j(\cdot): j\in\mJ)$. For states $\tilde{q}\in\tilde{\mQ}(n^{\kappa,(\infty)})$, let $\sigma_{\tilde{q}}$ define the first time $\bar{Q}^\kappa$ hits the state $\tilde{q}$.
As $\bar{Q}^{\kappa}$ is an irreducible, positive recurrent Markov chain, almost surely $\sigma_{\tilde{q}}<\infty$, $\forall\, \tilde{q}\in\tilde{Q}(n^{\kappa,(\infty)})$.

The packet level state of a multi-class queueing network with spinning behaves as a closed queueing network between arrival-departure events. Thus we can extend the packet level description of the $c$-th multi-class queueing network with spinning by defining, $\forall\, c>c'$
\begin{equation}\label{Ch2: Q coupling}
 \tilde{Q}^{(c)}(t)=\bar{Q}^{\kappa}(c(t-\tau^{\kappa, (c)}) + \sigma_{\tilde{q}^{\kappa, (c)}}),\qquad t\in(\tau^{\kappa, (c)}, \tau^{\kappa+1, (c)}).
\end{equation}
We will shortly define $\tau^{\kappa+1,(c)}$. 

We, also, define the flow level state of the $c$-th multi-class queueing network with spinning. We associate each packet in the closed queueing network $\bar{Q}^{\kappa}$ at time $\sigma_{\tilde{q}^{\kappa,(c)}}$ with a packet in the $c$-th MQNwS at time $\tau^{\kappa, (c)}$. For each route, let $k$ index the packets associated with each document at time $\tau^{\kappa,(c)}$. We retain this same index until time $\tau^{\kappa+1,(c)}$. Let $\bar{A}^{\kappa}_{ik}(t)$ denote the number of transitions where the $k$-th packet on route $i$ in the $c$-th MQNwS has traversed route $i$ in closed queueing network $\bar{Q}^{\kappa}(t)$ by time $t$. We define the components of the flow level process of the $c$-th MQNwS by
\begin{equation}\label{flow coupling 1}
 Y^{(c)}_{ik}(t)=Y_{ik}^{(c)}(\tau^{\kappa,(c)})-\bar{A}^{\kappa}_{ik}(c\{t-\tau^{\kappa,(c)}\} + \sigma_{\tilde{q}^{\kappa, (c)}}),
\end{equation}
for $k=1,...,n^{\kappa, (c)}_i,$ $i\in\mI$, $t\in (\tau^{\kappa,(c)}, \tau^{\kappa+1,(c)})$, and for $c>c'$. Similarly, for $c=\infty$, for the spinning network we define
\begin{equation}\label{flow coupling 2}
 Y^{(\infty)}_{ik}(t)=Y_{ik}^{(\infty)}(\tau^{\kappa,(\infty)})-\frac{\Lambda^{SN}_i(n^{\kappa, (\infty)})}{n_i^{\kappa,(\infty)}}(t-\tau^{k,(\infty)}),
\end{equation}
for $k=1,...,n^{\kappa, (\infty)}_i$, $i\in\mI$ and $t\in (\tau^{\kappa,(\infty)}, \tau^{\kappa+1,(\infty)})$.
Recall in the definition of the flow level state of a MQNwS, residual file sizes of each route are indexed to be increasing in size, (i.e. $y_{ik}\leq y_{i,k+1}$).
In both expressions (\ref{flow coupling 1}) and (\ref{flow coupling 2}) we do not index $Y_{ik}^{(c)}$ so that residual file sizes are increasing. Instead we index $Y_{ik}^{(c)}$ so that it is associated with a specific packet on route $i$ in closed queueing network $\bar{Q}^{\kappa}$. This representation is required so that packet indices do not change over interval $(\tau^{\kappa,(c)},\tau^{\kappa+1,(c)})$. Even so, these indices are a permutation of the ordering in which residual files sizes are increasing in size.

Note if the processing of the $(i,k)$-th document is not interrupted by another arrival departure event then, for $c>c'$, this document would depart at time
\begin{equation*}
 S_{ik}^{\kappa,(c)}=\tau^{\kappa, (c)}+\inf\lbrace t : \bar{A}^{\kappa,(c)}_{ik}(ct+\sigma_{\tilde{q}^{\kappa,(c)}})= Y_{ik}^{(c)}(\tau^{\kappa,(c)})\rbrace,
\end{equation*}
and, for $c=\infty$, this would occur at time
\begin{equation}
 S_{ik}^{\kappa,(\infty)}=\tau^{\kappa,(\infty)}+ \frac{n_i^{\kappa,(\infty)}}{\Lambda^{SN}_i(n^{\kappa,(\infty)})} Y_{ik}^{(\infty)}(\tau^{\kappa,(\infty)}). \label{Ch2: S infty}
\end{equation}
In addition, for each $i\in\mI$, let $E^\kappa_i$ be an independent exponential random variable with mean $\nu^{-1}_i$. We define $E_i^{\kappa, (c)}$ by
\begin{equation*}
 E_i^{\kappa,(c)}=\tau^{\kappa,(c)}+E^\kappa_i,\quad c\in\bN\cup \{ \infty\}.
\end{equation*}
 $E_i^{\kappa, (c)}$ denotes the next arrival of a $i$ document assuming it is uninterrupted by another arrival departure event.
From these terms, we can define the next arrival-departure event by 
\begin{equation}\label{Ch2: tau min}
 \tau^{\kappa+1,(c)}:=\min \bigg( \{S_{ik}^{\kappa, (c)} :\: k=1,...,N^{(c)}_i,\; i\in\mI \}\cup \{ E_i^{\kappa, (c)}:\: i\in\mI \} \bigg),\qquad c\in\bN\cup \{\infty\}
\end{equation}
which arrival-departure event occurs depends on which term minimises this term. Note by Assumption \ref{assump 3}, for each $c\in\bN\cup \{ \infty\}$ there is always a unique minimum of this term.

By the packet level coupling (\ref{Ch2: Q coupling}) and the flow level coupling (\ref{flow coupling 1}-\ref{flow coupling 2}), we have coupled our processes on the interval $(\tau^{\kappa,(c)},\tau^{\kappa+1,(c)})$. We now include the transition at time $\tau^{\kappa+1,(c)}$: if (\ref{Ch2: tau min}) is minimised by $S^{\kappa, (c)}_{ik}$, we define the $c$-th multi-class queueing network at time $\tau^{\kappa+1,(c)}$ by appropriately removing the $(i,k)$ document and packet from the network's state description at time $\tau^{\kappa+1,(c)}-$, and  if (\ref{Ch2: tau min}) is minimised by $ E_i^{\kappa, (c)}$ then we add a new document and packet to the flow state and packet state of the system, this document will be of (residual) size $X^{\kappa+1,(c)}_i \sim X^{(c)}_i$. Here $\{ X^{\kappa+1,(c)}_i \}_{c\in\bN\cup\{\infty\}}$ is an independent sequence of random variables satisfying,
\begin{equation}
 \frac{X_i^{\kappa+1,(c)}}{c}\xrightarrow[c\rightarrow\infty]{} X_i^{\kappa+1,(\infty)}. \label{Ch2: X kappa conv}
\end{equation}

This completes the coupling of our process on the interval $(\tau^{\kappa,(c)},\tau^{\kappa+1,(c)}]$.\\
\textbf{Proof of induction step:}\\
Given our coupling up to time $\tau^{\kappa+1,(c)}$, we now concern ourselves with proving the convergence statements (\ref{Ch2: time converge}) and (\ref{Ch2: space converge}). The following three lemmas will help to verify this.

\begin{lemmaa} \label{renewal lemma}Almost surely, for $i\in\mI$, $\eta>0$, $k=1,...,n_i$
\begin{equation*}
 \sup_{t\in [0,\eta]}\Bigg|\frac{\bar{A}^{\kappa}_{ik}(ct+\sigma_{\tilde{q}^{\kappa,(c)}})}{c}-\frac{\Lambda_i^{SN}(n^{\kappa,(c)})}{n^{\kappa,(c)}_i}t\Bigg|\xrightarrow[c\rightarrow \infty]{} 0.
\end{equation*}
\end{lemmaa}
\noindent \textit{Proof of Lemma A}.
We consider $c$ sufficiently large so that (\ref{Ch2: n to n}) holds. Let $R_{ik,\tilde{q}}(t)$ be the number times by time $t$ the $k$-th route $i$ packet has completed its route in the closed queueing network $\bar{Q}^{\kappa}$, when the closed queueing network was in state $\tilde{q}\in\tilde{\mQ}(n^{\kappa,(\infty)})$. Let $\gamma_{ik}(\tilde{q})$ be the drift of $R_{ik,\tilde{q}}$. For any Markov chain, the process that records the current state of the Markov chain and the next state, is also a Markov chain. So $R_{ik,\tilde{q}}$ is a renewal process and thus obeys the Functional Renewal Theorem. This gives that, almost surely, for all $\eta>0$
\begin{equation*}
 \max_{i\in\mI} \max_{k=1,..,n_i^{\kappa,(\infty)}} \max_{\tilde{q}\in\tilde{\mQ}(n^{\kappa,(\infty)})} \sup_{t\in [0,\eta]}\Bigg|\frac{R_{ik,\tilde{q}}(ct)}{c}-\gamma_{ik}(\tilde{q})t\Bigg|\xrightarrow[c\rightarrow \infty]{} 0.
\end{equation*}
For a proof of the Functional Renewal Theorem, see \cite[page 106]{ChYa01}. By the definition of $\bar{A}^{\kappa}_{ik}(t)$ and Corollary \ref{throughput}, we know that,
\begin{equation*}
 \bar{A}^{\kappa}_{ik}(t)=\sum_{\tilde{q}\in \tilde{\mQ}(n^{\kappa,(\infty)})}R_{ik,\tilde{q}}(t)\quad\text{and}\quad \frac{\Lambda_i^{SN}(n^{\kappa,(\infty)})}{n^{\kappa,(\infty)}_i}=\sum_{\tilde{q}\in \tilde{\mQ}(n^{\kappa,(\infty)})}\gamma_{ik}(\tilde{q}).
\end{equation*}
So, noting that $\tilde{\mQ}(n^{\kappa,(\infty)})$ is a finite set, we have that, almost surely, $\forall\, \eta >0$
\begin{align*}
&\sup_{t\in [0,\eta]}\Bigg| \frac{\bar{A}^{\kappa,(c)}_{ik}(ct)}{c}-\frac{\Lambda_i^{SN}(n^{\kappa,(\infty)})}{n^{\kappa,(\infty)}_i}t\Bigg|\\
&\leq \sum_{\tilde{q}\in \tilde{\mQ}(n^{\kappa,(\infty)})} \max_{r\in\mI} \max_{k'=1,...,n^{\kappa,(\infty)}_r}\sup_{t\in [0,\eta]}\Bigg| \frac{R_{rk',\tilde{q}}(ct)}{c}-\gamma_{rk'}(\tilde{q})t\Bigg|\xrightarrow[c\rightarrow\infty]{}0.
\end{align*}
As $\bar{Q}^{\kappa}$ is recurrent on all states in $\tilde{\mQ}(n^{\kappa,(\infty)})$, almost surely, $\sigma_{\tilde{q}} < \infty$ $\forall\, \tilde{q}\in \tilde{\mQ}(n^{\kappa,(\infty)})$. Thus, almost surely,
\begin{align*}
&\;\sup_{t\in [0,\eta]} \Bigg| \frac{\bar{A}^{\kappa}_{ik}(ct+\sigma_{\tilde{q}^{\kappa,(c)}})}{c} - \frac{\Lambda_i^{SN}(n^{\kappa,(\infty)})}{n^{\kappa,(\infty)}_i}\left(t+\frac{\sigma_{\tilde{q}^{\kappa,(c)}}}{c}-\frac{\sigma_{\tilde{q}^{\kappa,(c)}}}{c}\right) \Bigg|\\
&\leq \frac{\Lambda_i^{SN}(n^{\kappa,(\infty)})}{n^{\kappa,(\infty)}_i}\frac{\sigma_{\tilde{q}^{\kappa,(c)}}}{c} + \sup_{t\in[0,\eta+\sigma_{\tilde{q}^{\kappa,(c)}}]}\Bigg|\frac{\bar{A}^{\kappa}_{ik}(ct)}{c}-\frac{\Lambda_i^{SN}(n^{\kappa,(\infty)})}{n^{\kappa,(\infty)}_i}t\Bigg|\xrightarrow[c\rightarrow \infty]{} 0.
\end{align*}
\begin{flushright} \textit{QED Lemma A proven.} \end{flushright}
This renewal result characterises the limiting behaviour of $S^{\kappa,(c)}_{ik}$ (the time until document $k$'s departure given the current flow level state).

\begin{lemmab}\label{S conv} 
 For each $i\in\mI$ and $k=1,...,n_i^{\kappa,(\infty)}$, almost surely
\begin{equation*}
 S_{ik}^{\kappa,(c)}\xrightarrow[c\rightarrow \infty]{} S_{ik}^{\kappa,(\infty)}.
\end{equation*}
\end{lemmab}
\noindent \textit{Proof of Lemma B}.
By Lemma A and induction hypothesis (\ref{Ch2: Y kappa conv}), almost surely, $\forall\, \epsilon>0$ and $\forall\, \eta>0$ such that $\eta > S_{ik}^{\kappa,(\infty)}-\tau^{\kappa,(c)} +\frac{2\epsilon n^{\kappa,(\infty)}_i}{\Lambda_i^{SN}(n^{\kappa,(\infty)})}$, $\exists\, c'$ such that $\forall\, c>c'$,
\begin{gather}
\sup_{t\in [0,\eta]}\Bigg| \frac{\bar{A}^{\kappa,(\infty)}_{ik}(ct+\sigma_{\tilde{q}^{\kappa,(c)}}) }{c}-\frac{\Lambda_i^{SN}(n^{\kappa,(\infty)})}{n^{\kappa,(\infty)}_i}t\Bigg|<\epsilon,\\
\Bigg| \frac{Y_{ik}^{(c)}(\tau^{\kappa,(c)})}{c}- Y_{ik}^{(\infty)}(\tau^{\kappa,(\infty)}) \Bigg| < \epsilon.
\end{gather}
Hence, recalling the definition of $S_{ik}^{\kappa,(\infty)}$ in (\ref{Ch2: S infty}), the above two inequalities imply for all documents $(i,k)$
\begin{align*}
&\frac{1}{c} \bar{A}^{\kappa}_{ik}\left(c\bigg\{S_{ik}^{\kappa,(\infty)}-\tau^{\kappa,(\infty)}-\frac{2\epsilon n^{\kappa,(\infty)}_i}{\Lambda_i^{SN}(n^{\kappa,(\infty)})}\bigg\}+\sigma_{\tilde{q}^{\kappa,(c)}}\right)\\
&\leq Y_{ik}^{(\infty)}(\tau^{\kappa,(\infty)})-\epsilon
<\frac{Y_{ik}^{(c)}(\tau^{\kappa,(c)})}{c}.
\end{align*}
Thus
\begin{align*}
 S_{ik}^{(c)}-\tau^{\kappa,(c)}&=\inf\{ t\geq 0 : \bar{A}^{\kappa}_{ik}(ct+\sigma_{\tilde{q}^{\kappa,(c)}})=Y_{ik}^{(c)}(\tau^{\kappa,(c)})\}\\
 &> S_{ik}^{(\infty)}-\tau^{\kappa,(\infty)}-\frac{2\epsilon n^{\kappa,(\infty)}_i}{\Lambda_i^{SN}(n^{\kappa,(\infty)})}.
\end{align*}
By a similar argument one can see that
\begin{equation*}
S_{ik}^{(c)}-\tau^{\kappa,(c)}<S_{ik}^{(\infty)}-\tau^{\kappa,(\infty)}+\frac{2\epsilon n^{\kappa,(\infty)}_i}{\Lambda_i^{SN}(n^{\kappa,(\infty)})}.
\end{equation*}
Thus $S_{ik}^{(c)}-\tau^{\kappa,(c)}\rightarrow S_{ik}^{(\infty)}-\tau^{\kappa,(\infty)}$ as $c\rightarrow \infty$, almost surely. Thus as we assume (\ref{Ch2: tau kappa conv}) holds, almost surely 
\begin{equation*}
 S_{ik}^{(c)}\xrightarrow[c\rightarrow\infty]{} S_{ik}^{(\infty)}.
\end{equation*}
\begin{flushright} \textit{QED Lemma B proven.} \end{flushright}
Recall, that we had not ordered elements $Y_{ik}^{(c)}$ in increasing order, instead we indexed $Y_{ik}^{(c)}$ to be associated with each individual packet being processed in the closed queueing network $\bar{Q}^{\kappa}$. The following lemma helps us re-associate the desired increasing ordering of the terms $Y_{ik}^{(c)}$.

\begin{lemmac}
Let $y,y'\in\bR_+^n$ be such that $y_1\leq...\leq y_n$, $y'_1\leq ...\leq y'_n$ and let $p:\{1,...,n\}\rightarrow \{1,...,n\}$ be a permutation, then
\begin{equation}\label{Ch2: y max bound}
 \max_{k=1,...,n} |y_k-y'_k| \leq \max_{k=1,...,n} |y_k-y'_{p(k)}|.
\end{equation}
\end{lemmac}
\noindent \textit{Proof of Lemma C}.
We prove the result by induction on $n$, under the induction hypothesis that for all $\delta>0$
\begin{equation}\label{Ch2: y implies}
 |y_k-y'_{p(k)}|< \delta\:\: \forall\, k=1,...,n \quad\text{implies} \quad |y_k-y'_{k}|< \delta\:\: \forall\, k=1,...,n.
\end{equation}
The hypothesis clearly holds for $n=1$. Suppose the induction hypothesis holds for $n-1$. Take $i=p^{-1}(n)$ and $j=p(n)$. If $i=j(=n)$ then the problem clearly reduces to the $n-1$ case. Assume $i\neq j$. We know
\begin{equation*}
y_n\geq y_i \quad y'_n \geq y'_j.
\end{equation*}
Also, if
\begin{gather*}
 |y_n-y'_j|<\delta\quad \text{and}\quad |y_i-y'_n|<\delta
\end{gather*}
then
\begin{gather}
 y_n \geq y_i > y'_n -\delta\quad \text{and} \quad y'_n \geq y'_j > y_n -\delta\notag\\
\text{therefore}\:\: |y_n-y'_n| < \delta.\label{Ch2: y seq}
\end{gather}
Similarly,
\begin{gather}
 y_i > y'_n -\delta\geq y'_k \quad \text{and} \quad y'_j > y_n -\delta\geq y_i-\delta \notag\\
\text{therefore}\:\: |y_i-y'_j| < \delta. \label{Ch2: y no seq}
\end{gather}
We can now define a new permutation on $\{1,...,n-1\}$,
\begin{equation*}
p'(k)= 
\begin{cases}
p(k) & \text{if } k\neq i\\
j & \text{if } k=i.
 \end{cases}
\end{equation*}
Since (\ref{Ch2: y no seq}) holds we have reduced this problem to a problem on $n-1$ variables with equality (\ref{Ch2: y seq}) still holding. This completes the proof of our induction hypothesis. Since $\delta$ is arbitrary it is clear that (\ref{Ch2: y implies}) is equivalent to (\ref{Ch2: y max bound}).
\begin{flushright} \textit{QED Lemma C proven.} \end{flushright}

By Lemma B, induction hypothesis \eqref{Ch2: tau kappa conv} and the definition of $\tau^{k+1,(c)}$ (\ref{flow coupling 1}-\ref{flow coupling 2}), we know that $$\tau^{\kappa+1,(c)}\rightarrow \tau^{\kappa+1,(\infty)}\quad\text{ as }\quad c\rightarrow \infty.$$ By Assumption \ref{assump 3}, $\tau^{\kappa+1,(\infty)}$ is achieved by a distinct minimum and consequently there exists a $c''$ such that $\forall\, c>c''$ the argument attaining $\tau^{\kappa+1,(c)}$ in (\ref{Ch2: tau min}) is the same as that attaining $\tau^{\kappa+1,(\infty)}$. Thus the coupled processes will have the same document arrival-departure event occur at time $\tau^{\kappa+1,(c)}$, $\forall\, c\in\{c''+1,...,\infty\}$.

We can now verify (\ref{Ch2: time converge}) from the induction hypothesis:
\begin{align}
&\sup_{t\in (\tau^{\kappa,(c)},\tau^{\kappa+1,(c)}\wedge T]} |\lambda^{\kappa,(c)}(t)-t |\notag\\
&=|\tau^{\kappa,(\infty)}-\tau^{\kappa,(c)}|\vee|\tau^{\kappa+1,(\infty)}\wedge T-\tau^{\kappa+1,(c)}\wedge T|\xrightarrow[c\rightarrow\infty]{} 0. \label{Ch2: lambda conv}
\end{align}

We can also prove (\ref{Ch2: space converge}) from the induction hypothesis. We use the following set of inequalities which will subsequently be explained, $\forall\, c>c''$,
\begin{align}
 &\sup_{t\in(\tau^{\kappa,(c)},\tau^{\kappa+1,(c)}\wedge T]} \label{Ch2: line 1}
\bigg|\bigg| Y^{(\infty)}(\lambda^{\kappa,(c)}(t))-\frac{Y^{(c)}(t)}{c} \bigg|\bigg|\\
&\leq 
 \max_{i\in\mI} \max_{k=1,...,n_i^{\kappa,(\infty)}} \sup_{t\in (\tau^{\kappa,(c)},\tau^{\kappa+1,(c)}\wedge T]} \bigg|Y^{(\infty)}_{ik}(\lambda^{\kappa,(c)}(t))-\frac{Y^{(c)}_{ik}(t)}{c}\bigg| \label{Ch2: line 2}\\
&\qquad\quad + \max_{i\in\mI} \max_{k=1,...,n_i^{\kappa,(\infty)}}\bigg|X_{ik}^{\kappa+1,(\infty)}-\frac{X_{ik}^{\kappa+1,(c)}}{c}\bigg|\notag\\ 
&\leq \max_{i\in\mI} \max_{k=1,...,n_i^{\kappa,(\infty)}} \Bigg[ \bigg|Y^{(\infty)}_{ik}(\tau^{\kappa,(\infty)})-\frac{Y^{(c)}_{ik}(\tau^{\kappa,(c)})}{c}\bigg|\notag\\
&\qquad+ \frac{\Lambda^{SN}_i(n^{\kappa,(\infty)})}{n^{\kappa,(\infty)}_i}\times \sup_{t\in (\tau^{\kappa,(c)},\tau^{\kappa+1,(c)}\wedge T]}\big| \lambda^{\kappa,(c)}(t)-t \big|\notag\\
&\qquad + \sup_{t\in (\tau^{\kappa,(c)},\tau^{\kappa+1,(c)}]} \bigg| \frac{\Lambda^{SN}_i(n^{\kappa,(\infty)})}{n^{\kappa,(\infty)}_i}(t-\tau^{\kappa,(c)}) - \frac{\bar{A}^{\kappa}_{ik}(c\{t-\tau^{\kappa,(c)}\} + \sigma_{\tilde{q}^{\kappa,(c)}})}{c}\bigg|\notag\\
&\qquad\quad\quad +\frac{\Lambda_i^{SN}(n^{\kappa,(\infty)})}{n_i^{\kappa,(\infty)}} |\tau^{\kappa,(c)}-\tau^{\kappa,(\infty)}|\Bigg]\notag \\
&
\qquad \quad\qquad+\max_{i\in\mI} \max_{k=1,...,n_i^{\kappa,(\infty)}} \bigg|X_{ik}^{\kappa+1,(\infty)}-\frac{X_{ik}^{\kappa+1,(c)}}{c}\bigg| \label{Ch2: line 3}\\
&\qquad\quad\qquad \xrightarrow[c\rightarrow \infty]{} 0\notag
\end{align}
In the first inequality, we apply Lemma C so that we index each packet according to its position within closed queueing network $\bar{Q}^{\kappa}$ (see description of (\ref{flow coupling 1}) and (\ref{Ch2: Q coupling}). This is so residual file sizes $Y^{(c)}_{ik}(t)$ are not necessarily indexed to be increasing. Also the first inequality over estimates (\ref{Ch2: line 1}) by including the file sizes of all possible arrivals that could occur at time $\tau^{\kappa+1,(c)}$. In the second inequality we apply the triangle inequality to (\ref{Ch2: line 2}) by using the two facts,
\begin{align*}
 Y^{(\infty)}_{ik}(\lambda^{\kappa,(c)}(t))&=Y^{(\infty)}_{ik}(\tau^{\kappa,(\infty)})\\
&-\frac{\Lambda^{SN}_i(n^{\kappa,(\infty)})}{n^{\kappa,(\infty)}_i} \bigg( \{\lambda^{\kappa, (c)}(t)-t\} +\{t -\tau^{\kappa,(c)}\}+ \{\tau^{\kappa,(c)}- \tau^{\kappa,(\infty)}\} \bigg)\\
 Y^{(c)}_{ik}(t)&=Y^{(c)}_{ik}(\tau^{\kappa,(c)})-\bar{A}^{\kappa}_{ik}(c\{t-\tau^{\kappa,(c)}\} + \sigma_{\tilde{q}^{\kappa,(c)}}).
\end{align*}
for $t\in (\tau^{\kappa,(c)},\tau^{\kappa+1,(c)}\wedge T]$. The first expression in equation (\ref{Ch2: line 3}) converges to $0$ by induction assumption (\ref{Ch2: Y kappa conv}); the second expression converges to $0$ by (\ref{Ch2: lambda conv}); the third term converges by (\ref{Ch2: tau kappa conv}); the fourth converges by Lemma A and fifth term converges by expression (\ref{Ch2: X kappa conv}). We have thus demonstrated (\ref{Ch2: time converge}) and (\ref{Ch2: space converge}) hold. This verifies our induction hypothesis. 

Our induction argument is sufficient to couple our process on interval $[0,T]$. Since there are almost surely a finite number of documents in transfer at time $t=0$ and a finite number of document arrivals in interval $[0,T]$, it must be that
\begin{equation*}
 \{\kappa : \tau^{\kappa,(\infty)}<T \} \quad \text{is bounded almost surely.}
\end{equation*}
Since we have proven that almost surely $\tau^{\kappa,(c)}\rightarrow \tau^{\kappa, (\infty)}$ as $c\rightarrow \infty$, for all $\kappa\in\bZ_+$,
\begin{equation*}
  \{\kappa : \tau^{\kappa,(c)}<T \} \quad \text{is uniformly bounded over } c\in\bN\cup\{\infty\} \text{ almost surely.}
\end{equation*}
Thus by our inductive argument, we may couple our process $\{Y^{(c)}\}_{c\in\bN\cup\{\infty\}}$ on the interval $[0,T]$.

\noindent \textbf{Skorohod convergence:}\\
Taking $\lambda^{(c)}(t)=\lambda^{\kappa,(c)}(t)$ for $\, t\in [\tau^{\kappa,(c)}, \tau^{\kappa+1,(c)}\wedge T ]$. We have by statements (\ref{Ch2: time converge}) and (\ref{Ch2: space converge}) that
\begin{gather*}
 \sup_{t\in [0,T]} \big| \lambda^{(c)}(t) -t | \xrightarrow[c\rightarrow\infty]{} 0\\
\sup_{t\in [0,T]} \Big|\Big| Y^{(\infty)}(\lambda^{(c)}(t)) - \frac{Y^{(c)}(t)}{c} \Big|\Big| \xrightarrow[c\rightarrow\infty]{} 0.
\end{gather*}
Thus, almost surely, we have convergence in the Skorohod topology on $[0,T]$,
\begin{equation*}
 \frac{Y^{(c)}}{c}\xrightarrow[c\rightarrow \infty]{} Y^{(\infty)}.
\end{equation*}
Since, the Skorohod convergence occurs almost surely in this coupling, for all continuous bounded functions $f:D[0,t]\rightarrow \bR_+$,
\begin{equation*}
 f\bigg(\frac{Y^{(c)}}{c}\bigg)\xrightarrow[c\rightarrow \infty]{} f(Y^{(\infty)}), \qquad \text{almost surely}.
\end{equation*}
Thus by the Bounded Convergence Theorem,
\begin{equation*}
 \bE f\bigg(\frac{Y^{(c)}}{c}\bigg)\xrightarrow[c\rightarrow \infty]{} \bE f(Y^{(\infty)})
\end{equation*}
or, in other words, $\frac{Y^{(c)}}{c}$ converges weakly to $Y^{(\infty)}$ in the Skorohod topology. This completes the proof of Theorem \ref{spinning converge}.
\end{proof}

\section{Insensitivity of spinning networks}\label{Ch2: Insensitivity Sec}

The insensitivity of the spinning network is a consequence of Theorem \ref{spinning converge}. To prove this we will first require two technical lemmas.

\begin{lemma}\label{Ch2: X bar lemma}
Let $(X^{(c)}: c\in\bN\cup \{ \infty \})$ be a sequence of random variables. Let $X^{(c)}$ have values in $\bN$ and mean $\frac{c}{\mu}$ for $c\in \bN$ . Let $X^{(\infty)}$ have values in $\bR_+$ and mean $\frac{1}{\mu}$. Define random variables $(\bar{X}^{(c)}: c\in\bN\cup \{ \infty \})$ by
\begin{align}
 \bP(\bar{X}^{(c)}\leq y) &= \frac{\mu}{c} \sum_{z=1}^y \bP(X^{(c)}\geq z),\qquad c\in\bN,\notag \\
\bP( \bar{X}^{(\infty)} \leq y) &= \mu\int_0^y \bP(X^{(\infty)} \geq z) dz.\label{Ch2: X bar def}
\end{align}
If
\begin{equation*}
 \frac{X}{c}^{(c)} \Rightarrow X^{(\infty)} \quad \text{as} \quad c\rightarrow \infty,
\end{equation*}
then
\begin{equation}
 \frac{\bar{X}}{c}^{(c)} \Rightarrow \bar{X}^{(\infty)}\quad \text{as} \quad c\rightarrow \infty. \label{Ch2: X bar converges}
\end{equation}
\end{lemma}
\begin{proof}
$\bP( X^{(\infty)}\geq z)$ can only have countably many points of discontinuity. Thus by integration by substitution and the Bounded Convergence Theorem, we have that, for all $y\in\bR_+$,
\begin{align*}
 \;\;&\bP\Big(\frac{\bar{X}^{(c)}}{c}\leq y\Big)
 =\bP\Big(\bar{X}^{(c)}\leq \lfloor cy \rfloor\Big)
=  \frac{\mu}{c} \int^{\lceil cy \rceil}_1 \bP(X^{(c)} \geq z )dz\\
&= \mu\int^{\lceil cy \rceil/c}_{1/c} \bP\Big(\frac{X^{(c)}}{c} \geq z \Big)dz
 \xrightarrow[c\rightarrow\infty]{} \mu\int^{y}_{0} \bP(X^{(\infty)} \geq z )dz=\bP(\bar{X}^{(\infty)} \leq y).
\end{align*}
\end{proof}
To prove Theorem \ref{spinning converge}, we assumed no simultaneous arrival-departure events occurred. We now demonstrate that these assumptions hold for the case of a spinning network, with non-atomic document sizes.
\begin{lemma}\label{Ch2: no collision}
 Suppose the initial distribution $Y(0)$ conditional of $N(0)$ consists of independent non-atomic random variables $Y_{ik}$ $k=1,...,N_i(0), i\in\mI$. Given documents size distributions $X_i, i\in\mI$ are non-atomic, then, almost surely,\\
a) There are no simultaneous document arrival-departure events, i.e. Assumption \ref{assump 3} holds.\\
b) For all $t\in\bR_+$, almost surely, no document arrival-departure event occurs at time $t$.
\end{lemma}
\begin{proof}
Let $Y=(Y_t:t\in\bR_+)$ be the spinning networks flow level process description. Since arrivals $A_1,A_2,...$ form a Poisson process almost surely no two arrivals occur at the same time and for each $t\in\bR_+$ almost surely no arrival occurs at time $t$. Since exponential random variable $A_k-A_{k-1}$ is independent of $(Y_t: t\leq A_{k-1})$, there is zero probability that an arrival $A_k$ coincides with departures. Therefore, an arrival cannot coincide with a departure.

It remains to show that no two document departures may occur simultaneously. Let $D_k$ be the departure of some document $k$ of initial size $X_k$ (or initial residual size $Y_{ik}$ at time 0). Let $A_k$ be that document's arrival time (take $A_k=0$ if the document is present at time zero). Let $Y'$ be the process derived from $Y$ in which document $k$ never departs the SFLM (i.e. behaving as if $X_k=\infty$). Note that $Y'(t)=Y(t)$ for all $t<D_k$ and $D_k$ coincides with a document departure in $Y$ iff $D_k$ coincides with a departure in $Y'$. Note that as that $X_k$ is conditionally independent of $(Y'(t): t>A_k)$ conditional on $(Y'(t): t\leq A_k)$ and $D_k$ is non-atomic as it is a strictly increasing function of non-atomic independent random variable $X_k$. Thus, conditional on $(Y'(t): t\leq A_k)$ the probability that non-atomic random variable $D_k$ coincides with the countable set of departure events in $(Y'(t): t>A_k)$ or at a specific time $t\in\bR_+$ is zero. Thus, the probability two departure events coincide is zero and the probability that departure occurs at a specific time $t$ is zero.
\end{proof}

We can now prove one of the main results of this chapter: the insensitivity of the spinning network. 

\begin{corollary}\label{Ch2: insensitive coroll}
Given Assumption \ref{MQNwS assump}, 
the spinning network has a stationary distribution which is insensitive to all non-atomic document size distributions.
\end{corollary}

\begin{proof}
We can take document sizes $X^{(c)}_i$ such that $\bE X^{(c)}_i = \frac{c}{\mu_i}$ and $\frac{X_i}{c}^{(c)}\Rightarrow X_i^{(\infty)}$.
As in Theorem \ref{spinning converge}, we consider a sequence of multi-class queueing networks with spinning associated with these document size distributions and with queue service rates $c\phi_j(\cdot)$.
From Corollary \ref{CH1: N ed thrm} and Corollary \ref{Ch2: packet flow corol}, the prelimit stationary distribution of $Y^{(c)}$ and $N^{(c)}$, $c\in\bN$ is
\begin{align}
&\bP(Y^{(c)}(0)=y)\notag\\
&=\frac{B_n}{B}\prod_{i\in\mI}\left( \begin{array}{cc} n_i \\ n_{iy}\: : \: y\in\bN \end{array}\right) \prod_{k=1}^{n_i} \big(\nu_i \bP(X_i\geq y_{ik})\big),\qquad \forall\, y\in\mY,\; t\in\bR_+,\\
&\bP(N^{(c)}(0)=n)=\frac{B_n}{B} \prod_{i\in\mI} \rho_i^{n_i},\qquad \forall\,\;n\in\bZ_+^I. \label{Ch2: N insens}
\end{align}
We can construct $Y^{(c)}(0)$, by taking a vector $N(0)$ according to distribution (\ref{Ch2: N insens}) then, for each $i\in\mI$ and $k=1,...,N_i(0)$, $Y_{ik}^{(c)}(0)$ is taken by selecting and ordering independent random variables $\bar{X}^{(c)}_{i}$, where $\bar{X}_i^{(c)}$ is defined from $X_i^{(c)}$ by (\ref{Ch2: X bar def}). Given Lemma \ref{Ch2: X bar lemma}, 
\begin{equation}\label{Y conv insens}
\frac{Y^{(c)}(0)}{c}\Rightarrow Y^{(\infty)}(0)\qquad\text{as}\qquad c \rightarrow \infty, 
\end{equation}
where $Y^{(\infty)}(0)$ has density
\begin{equation}\label{Ch2: gsflm density}
 \frac{B_n}{B} \prod_{i\in\mI} n_i!\prod_{k=1}^{n_i} \big(\rho_i  \bP(\bar{X}^{(\infty)}_i\in dx_{ik})\big),
\end{equation}
and also 
\begin{equation}\label{Ch2: N insens 2}
bP(N^{(\infty)}(0)=n)=\frac{B_n}{B} \prod_{i\in\mI} \rho_i^{n_i},
\end{equation}
$\forall\,\;n\in\bZ_+^I$ and for $x_{ik}\leq x_{ik+1}$ $k=1,...,n_i-1$, $i\in\mI$. Note the above expression for $N^{(\infty)}$ depends on $X^{(\infty)}_i$ only through its mean. Hence if \eqref{Ch2: gsflm density} is the stationary distribution for our limit process then this distribution must be insensitive.

We now show that \eqref{Ch2: gsflm density} provides a stationary distribution.
It is known that if a sequence of processes weakly converge in the Skorohod topology and if, almost surely, there is not jump at time $t$ then the marginal distribution at time $t$ must weakly converge, see \cite[Theorem 12.5]{Bi99}.
By  Lemma \ref{Ch2: no collision}, almost surely no jump occurs at time $t$ for $N^{(\infty)}$ and, by Theorem \ref{spinning converge}, $N^{(c)}\Rightarrow N^{(\infty)}$ as $c\rightarrow\infty$ in the Skorohod topology. Thus, $N^{(c)}(t)\Rightarrow N^{(\infty)}(t)$ i.e. the marginal distributions converge at time $t$. 
Thus, when processes $Y^{(c)}$, $c\in\bN$ are stationary, by for any continuous bounded function $f:\mY\rightarrow \bR$
\begin{equation*}
 \bE f(Y^{(\infty)}(0)) = \lim_{c\rightarrow \infty}  \bE f\Big(\frac{Y^{(c)}(0)}{c}\Big) = \lim_{c\rightarrow \infty}  \bE f\Big(\frac{Y^{(c)}(t)}{c}\Big) =  \bE f(Y^{(\infty)}(t)).
\end{equation*}
The first equality holds by \eqref{Y conv insens}; the second holds by the stationarity of $Y^{(c)}$; and the third holds by the weak convergence of the marginal distributions.
This proves (\ref{Ch2: gsflm density}) gives a stationary distribution of the spinning network, and consequently, from (\ref{Ch2: N insens}) and (\ref{Ch2: N insens 2}), we see that the spinning network is insensitive.
\end{proof}


\bibliography{references}
\bibliographystyle{ims}

\end{document}